\documentclass[12pt,reqno]{article}
\UseRawInputEncoding

\usepackage[top=2.5cm,bottom=2.5cm,left=2.5cm,right=2.5cm]{geometry}
\usepackage{graphicx}
\usepackage{latexsym}
\usepackage[top=2.5cm,bottom=2.5cm,left=2.5cm,right=2.5cm]{geometry}
\usepackage{graphicx}
\usepackage{latexsym}
\usepackage{amsmath,amssymb}
\usepackage{amscd}

\usepackage[arrow,matrix]{xy}
\usepackage{graphicx}
\usepackage{xcolor}
\usepackage{times}
\usepackage[]{subfigure}
 \usepackage{cite}
 \usepackage{float}

\usepackage{hyperref}

\usepackage[colorinlistoftodos]{todonotes}
\newcommand\on{\operatorname}

\newcommand\Ric{\on{Ric}}

\usepackage{amscd,amsmath,amsthm,amssymb}

\theoremstyle{plain}
\newtheorem{theorem}{Theorem}[section]
\newtheorem{proposition}[theorem]{Proposition}
\newtheorem{corollary}[theorem]{Corollary}

\theoremstyle{definition}
\newtheorem{remark}[theorem]{Remark}

\newtheorem{example}[theorem]{Example}

\begin{document}

\title{ Ricci-Yamabe Soliton on a Class of $4$-Dimensional Walker Manifolds}
\author {{Abdou Bousso$^{1}$}\thanks{{
 E--mail: \texttt{abdoukskbousso@gmail.com} (A. Bousso)}},\texttt{ }Ameth  Ndiaye$^{2}$\footnote{{
 E--mail: \texttt{ameth1.ndiaye@ucad.edu.sn} (A. Ndiaye)}}\\
\begin{small}{$^{1}$D\'epartement de Math\'ematiques et Informatique, FST, Universit\'e Cheikh Anta Diop,  \.Dakar, S\'en\'egal.}\end{small}\\ 
\begin{small}{$^{2}$D\'epartement de Math\'ematiques, FASTEF, Universit\'e Cheikh Anta Diop, Dakar, Senegal.}\end{small}}
\date{}
\maketitle%


\begin{abstract} 
 This article explores Ricci-Yamabe solitons on a specific class of 4-dimensional Walker manifolds. Walker manifolds, characterized by the existence of a parallel null distribution, find applications in general relativity and are fundamental objects of geometric study. We consider a particular pseudo-Riemannian metric, which depends on the smooth functions $f_1, f_2, f_3$. The main objective is to determine the conditions under which this manifold admits a Ricci-Yamabe soliton. We will explicitly calculate the components of the Ricci tensor, the scalar curvature, and the components of the Hessian Perelman potential. Solving the resulting system of partial differential equations, we will identify the constraints on the functions $f_1, f_2, f_3$ and the vector field $X$ for the existence of such solitons. Specific examples and their geometric properties will also be discussed.
\end{abstract}
\begin{small} {\textbf{MSC:}  53C20, 53C21.}
\end{small}\\
\begin{small} {\textbf{Keywords:} Ricci soliton, Ricci-Yamabe soliton, Riemannian metric, Walker manifold.} 
\end{small}\\
\maketitle

\section{Introduction}
Differential geometry is a fundamental field of mathematics, providing the necessary tools for the study of shapes and spaces. At the heart of this discipline are the concepts of Riemannian and pseudo-Riemannian manifolds, equipped with metrics that define the notions of distance, angle, and curvature. Understanding curvature, encoded by the Ricci tensor ($R_{ij}$) and the scalar curvature ($R$), is essential for characterizing the local and global geometry of a manifold.

A particularly dynamic aspect of modern geometry is the study of geometric flows, such as the famous Ricci flow introduced by Hamilton \cite{Hamilton1982}, or the Yamabe flow \cite{Yamabe1960}. These metric evolution equations have led to major advancements, notably in the proof of the Poincar\'e conjecture by Perelman. Geometric solitons emerge as self-similar solutions or fixed points of these flows, thus playing a crucial role in understanding their long-term behavior \cite{Cao2010}.

In this context, Ricci-Yamabe solitons constitute a class of generalized solitons, combining aspects of Ricci and Yamabe flows. A pseudo-Riemannian manifold $(M, g)$ is called a Ricci-Yamabe soliton if there exists a smooth vector field $X$ over $M$ and real constants $\lambda$, $\beta_1$, and $\beta_2$ such that the following equation is satisfied:
\begin{eqnarray}\label{Yamabe}
2\beta_1Ric +  \mathcal{L}_X g + (-\beta_2R + 2\lambda) g = 0,
\end{eqnarray}
where $(\mathcal{L}_X g)_{ij}$ is the Lie derivative of the metric $g$ with respect to the vector field $X$. The constant $\lambda$ determines the nature of the soliton: it is dilating if $\lambda > 0$, contracting if $\lambda < 0$, and steady if $\lambda = 0$. These solitons are of great interest because they generalize other important structures; for example, if the scalar curvature $R$ is constant and equal to $\lambda$, the equation reduces to that of a Ricci soliton \cite{Petersen2006}. Moreover, if the vector field $X$ is zero, the manifold is an Einstein manifold, characterized by a Ricci tensor proportional to the metric. The study of these solitons has progressed significantly on various classes of varieties \cite{Besse1987}.

At the same time, Walker manifolds form a fascinating class of pseudo-Riemannian manifolds, which play an important role in general relativity. They are defined by the existence of a parallel zero distribution, which makes them suitable for modeling special spacetimes, in particular those that admit gravitational waves or particular zero vector fields \cite{Pravda2002, Hall2004}. Their specific algebraic and geometric structure provides a rich framework for the exploration of new solutions for fundamental equations in physics and geometry.

In this article, we propose to study the Ricci-Yamabe solitons on a specific class of 4-dimensional Walker manifolds. The metric we consider is given by:
$$g = \begin{pmatrix} 0 & 0 & 1 & 0 \\ 0 & 0 & 0 & 1 \\ 1 & 0 & f_1(x,y,u,v) & f_2(x,y,u,v) \\ 0 & 1 & f_2(x,y,u,v) & f_3(x,y,u,v) \end{pmatrix}$$
where $(x, y, u, v)$ are the local coordinates and $f_1, f_2, f_3$ are smooth functions of the coordinates. This metric form represents an interesting particular case of Walker manifolds, which has not been extensively explored from the perspective of Ricci-Yamabe solitons to our knowledge.

The main objective of this study is to determine the conditions on the functions $f_1, f_2, f_3$ and the components of the vector field $X$ for the manifold $(M, g)$ to admit a Ricci-Yamabe soliton. To do this, we will proceed in several steps. In Section \ref{sec:preliminaires}, we will establish the notations and recall the key definitions and we will be devoted to the computations of the Christoffel symbols, the Ricci tensor, the scalar curvature, the Laplace-Beltrami operator of the and the Hessian of the Perelman pontential. In the Section \ref{sec:resultats}, we will set up and analyze the system of partial differential equations resulting from the Ricci-Yamabe soliton condition, looking for exact solutions or conditions of existence.
\section{Preliminaries and Notations}
\label{sec:preliminaires}
\subsection{Walker $4$-manifold}
In this section, we introduce the essential concepts and notations that will be used throughout the article. \\
We consider a differentiable variety of dimension 4, denoted $M$, equipped with a system of local coordinates $(x, y, u, v)$. The pseudo-Riemannian metric $g$ on $M$ is given by the matrix of its components $g_{ij}$ in this coordinate system.

The pseudo-Riemannian metric $g$ in the manifold $M$ that we are studying is a particular case of a Walker metric in dimension 4. Its components are given by:
\begin{eqnarray}\label{metric}g_{ij} = \begin{pmatrix}
0 & 0 & 1 & 0 \\
0 & 0 & 0 & 1 \\
1 & 0 & f_1(x,y,u,v) & f_2(x,y,u,v) \\
0 & 1 & f_2(x,y,u,v) & f_3(x,y,u,v)
\end{pmatrix}
\end{eqnarray}
where the functions $f_1, f_2, f_3$ are smooth real functions on the variety $M$ depending on the four coordinates $(x, y, u, v)$.
For the calculations of curvature quantities, it is essential to have the components of the inverse metric, denoted $g^{ij}$, such as $g^{ik} g_{kj} = \delta^i_j$ (where $\delta^i_j$ is the Kronecker symbol). For the given metric $g_{ij}$, its inverse matrix is easily calculable :
$$g^{ij} = \begin{pmatrix}
-f_1 & -f_2 & 1 & 0 \\
-f_2 & -f_3 & 0 & 1 \\
1 & 0 & 0 & 0 \\
0 & 1 & 0 & 0
\end{pmatrix}.$$
The Levi-Civita connection $\nabla$ associated with the metric $g$ is determined by the Christoffel symbols $\Gamma^k_{ij}$. These are defined by the formula:
$$\Gamma^k_{ij} = \frac{1}{2} g^{kl} \left( \frac{\partial g_{il}}{\partial x^j} + \frac{\partial g_{jl}}{\partial x^i} - \frac{\partial g_{ij}}{\partial x^l} \right)$$
where $\frac{\partial}{\partial x^i}$ represents the partial derivative with respect to the $i$-th coordinate ($x^1=x, x^2=y, x^3=u, x^4=v$). The explicit calculation of these symbols is the first crucial technical step. They will depend on the functions $f_1, f_2, f_3$ and their first partial derivatives. By using the above formula, the only non zero components of $\Gamma_{ij}^k$ are :

\[
\begin{gathered}
\Gamma^1_{13}=\Gamma^1_{31}=\tfrac12\,\partial_x f_1,\quad
\Gamma^1_{14}=\Gamma^1_{41}=\tfrac12\,\partial_x f_2,\quad
\Gamma^1_{23}=\Gamma^1_{32}=\tfrac12\,\partial_y f_1,\quad
\Gamma^1_{24}=\Gamma^1_{42}=\tfrac12\,\partial_y f_2,\\
\Gamma^1_{33}=\tfrac12\!\big(f_1\,\partial_x f_1+f_2\,\partial_y f_1+\partial_u f_1\big),
\Gamma^1_{34}=\Gamma^1_{43}=\tfrac12\!\big(f_1\,\partial_x f_2+f_2\,\partial_y f_2+\partial_v f_1\big),\\
\Gamma^1_{44}=\tfrac12\!\big(f_1\,\partial_x f_3+f_2\,\partial_y f_3\big)
+\partial_v f_2-\tfrac12\,\partial_u f_3,
\Gamma^2_{13}=\Gamma^2_{31}=\tfrac12\,\partial_x f_2,\quad
\Gamma^2_{14}=\Gamma^2_{41}=\tfrac12\,\partial_x f_3,\\
\Gamma^2_{23}=\Gamma^2_{32}=\tfrac12\,\partial_y f_2,\quad
\Gamma^2_{24}=\Gamma^2_{42}=\tfrac12\,\partial_y f_3,
\Gamma^2_{33}=\tfrac12\!\big(f_2\,\partial_x f_1+f_3\,\partial_y f_1+2\,\partial_u f_2-\partial_v f_1\big),\\
\Gamma^2_{34}=\Gamma^2_{43}=\tfrac12\!\big(f_2\,\partial_x f_2+f_3\,\partial_y f_2+\partial_u f_3\big),
\Gamma^2_{44}=\tfrac12\!\big(f_2\,\partial_x f_3+f_3\,\partial_y f_3+\partial_v f_3\big),\\
\Gamma^3_{33}=-\tfrac12\,\partial_x f_1,\quad
\Gamma^3_{34}=\Gamma^3_{43}=-\tfrac12\,\partial_x f_2,\quad
\Gamma^3_{44}=-\tfrac12\,\partial_x f_3,\\
\Gamma^4_{33}=-\tfrac12\,\partial_y f_1,\quad
\Gamma^4_{34}=\Gamma^4_{43}=-\tfrac12\,\partial_y f_2,\quad
\Gamma^4_{44}=-\tfrac12\,\partial_y f_3.
\end{gathered}
\]

The Ricci tensor $R_{ij}$ of the metric $g$ is a contraction of the Riemann curvature tensor. For a Levi-Civita metric, its components are given by:
$$R_{ij} = \partial_k \Gamma^k_{ij} - \partial_j \Gamma^k_{ik} + \Gamma^k_{ij} \Gamma^l_{kl} - \Gamma^l_{ik} \Gamma^k_{jl}.$$
The only components of the Ricci tensor that are not zero are \[
\begin{aligned}
R_{13}=R_{31} &= \tfrac12\,\partial_x^2 f_1 + \tfrac12\,\partial_x\partial_y f_2,\quad
R_{23}=R_{32} &= \tfrac12\,\partial_y^2 f_2 + \tfrac12\,\partial_x\partial_y f_1,\\[4pt]
R_{14}=R_{41} &= \tfrac12\,\partial_x^2 f_2 + \tfrac12\,\partial_x\partial_y f_3,\quad
R_{24}=R_{42} &= \tfrac12\,\partial_y^2 f_3 + \tfrac12\,\partial_x\partial_y f_2.
\end{aligned}
\]

$$\begin{aligned}
R_{33} &= \tfrac12\,f_1\,\partial_x^2 f_1
       + f_2\,\partial_x\partial_y f_1
       + \tfrac12\,f_3\,\partial_y^2 f_1\\
      &\quad + \tfrac12\,(\partial_x f_1)(\partial_y f_2)
       -\tfrac12\,(\partial_y f_1)(\partial_x f_2)
       +\tfrac12\,(\partial_y f_1)(\partial_y f_3)
       -\tfrac12\,(\partial_y f_2)^2\\
      &\quad - \partial_y\partial_v f_1 + \partial_y\partial_u f_2,
\\[6pt]
R_{34} =R_{43}&= \tfrac12\,f_1\,\partial_x^2 f_2
       + f_2\,\partial_x\partial_y f_2
       + \tfrac12\,f_3\,\partial_y^2 f_2\\
      &\quad -\tfrac12\,(\partial_y f_1)(\partial_x f_3)
       +\tfrac12\,(\partial_x f_2)(\partial_y f_2)\\
      &\quad +\tfrac12\,\partial_x\partial_v f_1
       -\tfrac12\,\partial_x\partial_u f_2
       -\tfrac12\,\partial_y\partial_v f_2
       +\tfrac12\,\partial_y\partial_u f_3,
\\[6pt]
R_{44} &= \tfrac12\,f_1\,\partial_x^2 f_3
       + f_2\,\partial_x\partial_y f_3
       + \tfrac12\,f_3\,\partial_y^2 f_3\\
      &\quad +\tfrac12\,(\partial_x f_1)(\partial_x f_3)
       -\tfrac12\,(\partial_x f_2)^2
       +\tfrac12\,(\partial_x f_2)(\partial_y f_3)
       -\tfrac12\,(\partial_y f_2)(\partial_x f_3)\\
      &\quad + \partial_x\partial_v f_2 - \partial_x\partial_u f_3.
\end{aligned}$$
The scalar curvature $R$ is obtained by contracting the Ricci tensor with the inverse metric: $$R = g^{ij} R_{ij}.$$ It is a scalar function that summarizes the average curvature at each point of the manifold. A direct calculation gives us the the scalar curvature of the metric $g$ by
\begin{align*}
R &= g^{11}R_{11} + g^{22}R_{22} + 2g^{12}R_{12} + 2g^{13}R_{13} + 2g^{14}R_{14} + 2g^{23}R_{23} + 2g^{24}R_{24} \\
&\quad + g^{33}R_{33} + 2g^{34}R_{34} + g^{44}R_{44} \\
&= (-f_1)R_{11} + (-f_3)R_{22} + 2(-f_2)R_{12} + 2(1)R_{13} + 2(0)R_{14} + 2(0)R_{23} + 2(1)R_{24} \\
&\quad + (0)R_{33} + 2(0)R_{34} + (0)R_{44} \\
&=  2R_{13} + 2R_{24}=\partial xx f_1+\partial yy f_3 +2\partial x\partial y f_2.
\end{align*}
\begin{remark}
If $(M^4,g)$ is a strict Walker manifold then \(\operatorname{Ricci} = 0\).
\end{remark}
The Laplace-Beltrami operator $\Delta$  of a function $f$ is defined by the formula:
\begin{eqnarray}\label{Laplacien1}
\Delta f = \frac{1}{\sqrt{|g|}} \partial_i (\sqrt{|g|} g^{ij} \partial_j f).
\end{eqnarray}
With the metric $g$, we have $|g| = 1$, which simplifies the operator to :
\begin{eqnarray}\label{Laplacien2}
\Delta f = \partial_i (g^{ij} \partial_j f) = g^{ij} \partial_i \partial_j f + (\partial_i g^{ij}) \partial_j f.
\end{eqnarray}
Using the formula \eqref{Laplacien2}, we have 
\begin{align*}
\Delta f &= \left(-f_1 \frac{\partial^2 f}{\partial x^2} - 2f_2 \frac{\partial^2 f}{\partial x \partial y} - f_3 \frac{\partial^2 f}{\partial y^2} + 2 \frac{\partial^2 f}{\partial x \partial u} + 2 \frac{\partial^2 f}{\partial y \partial v}\right) \\
&- \left(\frac{\partial f_1}{\partial x} + \frac{\partial f_2}{\partial y}\right) \frac{\partial f}{\partial x} - \left(\frac{\partial f_2}{\partial x} + \frac{\partial f_3}{\partial y}\right) \frac{\partial f}{\partial y}.
\end{align*}
The  Hessian operator of a function $f$ is defined by $(\nabla^2 f)_{ij} = \frac{\partial^2 f}{\partial x^i\ \partial x^j} - \Gamma^k_{ij} \frac{\partial f}{\partial x^k}$.
Using the Christoffel symbols calculated above, we obtain the  non zero components of the Hessian :

\begin{eqnarray}\label{Hessien}
     (\nabla^2 f)_{11} & =& \frac{\partial^2 f}{\partial x^2} - \left(\Gamma^1_{11}\frac{\partial f}{\partial x}+\Gamma^2_{11}\frac{\partial f}{\partial y}+\Gamma^3_{11}\frac{\partial f}{\partial u} + \Gamma^4_{11}\frac{\partial f}{\partial v}\right) = \frac{\partial^2 f}{\partial x^2}\nonumber\\
    (\nabla^2 f)_{12} &=& \frac{\partial^2 f}{\partial x \partial y} - \left(\Gamma^1_{12}\frac{\partial f}{\partial x} +\Gamma^2_{12}\frac{\partial f}{\partial y} +\Gamma^3_{12}\frac{\partial f}{\partial u} + \Gamma^4_{12}\frac{\partial f}{\partial v}\right) =\frac{\partial^2 f}{\partial x \partial y}\nonumber\\
     (\nabla^2 f)_{13} &=& \frac{\partial^2 f}{\partial x \partial u} - \left(\Gamma^1_{13}\frac{\partial f}{\partial x} +\Gamma^2_{13}\frac{\partial f}{\partial x} +\Gamma^3_{13}\frac{\partial f}{\partial u} + \Gamma^4_{13}\frac{\partial f}{\partial v}\right) = \frac{\partial^2 f}{\partial x \partial u}-\frac{1}{2}\frac{\partial f_1}{\partial x}\frac{\partial f}{\partial x}-\frac{1}{2}\frac{\partial f_2}{\partial x}\frac{\partial f}{\partial y} \nonumber\\
     (\nabla^2 f)_{14} &=& \frac{\partial^2 f}{\partial x \partial v} - \left(\Gamma^1_{14}\frac{\partial f}{\partial x} +\Gamma^2_{14}\frac{\partial f}{\partial y} +  \Gamma^3_{14}\frac{\partial f}{\partial u} + \Gamma^4_{14}\frac{\partial f}{\partial v}\right) = \frac{\partial^2 f}{\partial x \partial v}-\frac{1}{2}\frac{\partial f_2}{\partial x}-\frac{1}{2}\frac{\partial f_3}{\partial x}\nonumber\\
    (\nabla^2 f)_{22} &=& \frac{\partial^2 f}{\partial y^2} - \left(\Gamma^1_{22}\frac{\partial f}{\partial x} + \Gamma^2_{22}\frac{\partial f}{\partial y} + \Gamma^3_{22}\frac{\partial f}{\partial u} + \Gamma^4_{22}\frac{\partial f}{\partial v}\right) = \frac{\partial^2 f}{\partial y^2}\nonumber\\
     (\nabla^2 f)_{23} &=& \frac{\partial^2 f}{\partial y \partial u} - \left(\Gamma^1_{23}\frac{\partial f}{\partial x} +\Gamma^2_{23}\frac{\partial f}{\partial y} +\Gamma^3_{23}\frac{\partial f}{\partial u} + \Gamma^4_{23}\frac{\partial f}{\partial v}\right) = \frac{\partial^2 f}{\partial y \partial u}-\frac{1}{2}\frac{\partial f_1}{\partial y}\frac{\partial f}{\partial x}-\frac{1}{2}\frac{\partial f_2}{\partial y}\frac{\partial f}{\partial y}\nonumber\\
     (\nabla^2 f)_{24} &=& \frac{\partial^2 f}{\partial y \partial v} - \left(\Gamma^1_{24}\frac{\partial f}{\partial x} +\Gamma^2_{24}\frac{\partial f}{\partial y} +\Gamma^3_{24}\frac{\partial f}{\partial u} + \Gamma^4_{24}\frac{\partial f}{\partial v}\right) = \frac{\partial^2 f}{\partial y \partial v}-\frac{1}{2}\frac{\partial f_2}{\partial y}\frac{\partial f}{\partial x}-\frac{1}{2}\frac{\partial f_3}{\partial y}\frac{\partial f}{\partial y}\nonumber\\
     (\nabla^2 f)_{33} &=& \frac{\partial^2 f}{\partial u^2} - \left(\Gamma^1_{33}\frac{\partial f}{\partial x} +\Gamma^2_{33}\frac{\partial f}{\partial y} +\Gamma^3_{33}\frac{\partial f}{\partial u} + \Gamma^4_{33}\frac{\partial f}{\partial v}\right) \nonumber\\
     &=& \frac{\partial^2 f}{\partial u^2} -\tfrac12\big(f_1\partial_x f_1 + f_2\partial_y f_1 + \partial_u f_1\big) \frac{\partial f}{\partial x}-\tfrac12\big(f_2\partial_x f_1+f_3\partial_y f_1+2\partial_u f_2-\partial_v f_1\big)\frac{\partial f}{\partial y}\nonumber\\
     && -\frac{1}{2}\left(\partial_x f_1 \frac{\partial f}{\partial u} + \partial_y f_1 \frac{\partial f}{\partial v}\right)\nonumber\\
     (\nabla^2 f)_{34} &=& \frac{\partial^2 f}{\partial u \partial v} - \left(\Gamma^1_{34}\frac{\partial f}{\partial x} + \Gamma^2_{34}\frac{\partial f}{\partial y} + \Gamma^3_{34}\frac{\partial f}{\partial u} + \Gamma^4_{34}\frac{\partial f}{\partial v}\right)\nonumber\\
     &=& \frac{\partial^2 f}{\partial u \partial v} -\tfrac12\big(f_2\partial_x f_2+f_3\partial_y f_2+\partial_u f_3\big)\frac{\partial f}{\partial x}-\tfrac12\big(f_2\partial_x f_2+f_3\partial_y f_2+\partial_u f_3\big)\frac{\partial f}{\partial y}\nonumber\\
     &&- \frac{1}{2}\left(\partial_x f_2 \frac{\partial f}{\partial u} + \partial_y f_2 \frac{\partial f}{\partial v}\right)\nonumber\\
     (\nabla^2 f)_{44} &=& \frac{\partial^2 f}{\partial v^2} - \left(\Gamma^1_{44}\frac{\partial f}{\partial x} +\Gamma^2_{44}\frac{\partial f}{\partial y} +\Gamma^3_{44}\frac{\partial f}{\partial u} + \Gamma^4_{44}\frac{\partial f}{\partial v}\right)\nonumber\\
     &= &\frac{\partial^2 f}{\partial v^2}-\left(\tfrac12\big(f_1\partial_x f_3 + f_2\partial_y f_3\big)+\partial_v f_2-\tfrac12\partial_u f_3\right)\frac{\partial f}{\partial x}\nonumber\\
     &&-\tfrac12\big(f_2\partial_x f_3+f_3\partial_y f_3+\partial_v f_3\big)\frac{\partial f}{\partial y}- \frac{1}{2}\left(\partial_x f_3 \frac{\partial f}{\partial u} + \partial_y f_3 \frac{\partial f}{\partial v}\right)
\end{eqnarray}
\begin{remark}
    Let's choose a metric \(\widetilde{g}=g\) such that \(f_2(x,y,u,v)=0\) and \(f_1(x,y,u,v)=f_3(x,y,u,v)=xa(u,v)+yb(u,v)+c(u,v)\). In this case \(M^4,\widetilde{g})\) is a Walker  \(4\)-manifold. Moreover by substitution of $f_2=0$ and $f_1=f_3=x\,a(u,v)+y\,b(u,v)+c(u,v)$ (with $a,b,c$ depend only of $u,v$), we have

$$
\partial_x f_1=\partial_x f_3=a,\quad
\partial_y f_1=\partial_y f_3=b,\quad
\partial_x^2 f_1=\partial_y^2 f_1=\partial_x\partial_y f_1=0,
$$

$$
\partial_x\partial_v f_1=a_v,\quad
\partial_y\partial_v f_1=b_v,\quad
\partial_y\partial_u f_3=b_u,\quad
\partial_x\partial_u f_3=a_u,
$$
and all the derivatives of $f_2$ are zero.

Thus, the Ricci matrix becomes

$$
\operatorname{Ric}=(R_{ij})=
\begin{pmatrix}
0 & 0 & 0 & 0\\[4pt]
0 & 0 & 0 & 0\\[4pt]
0 & 0 & \tfrac12\,b^2-b_v & -\tfrac12\,a b+\tfrac12\,a_v+\tfrac12\,b_u\\[4pt]
0 & 0 & -\tfrac12\,a b+\tfrac12\,a_v+\tfrac12\,b_u & \tfrac12\,a^2-a_u
\end{pmatrix},
$$

where $a=a(u,v)$, $b=b(u,v)$ and the index $u,v$ denote the partial derivatives ($a_u=\partial_u a$, $a_v=\partial_v a$, etc.).

and all the other components are null and \(R=0\). Then we have
$$
\begin{aligned}
(\nabla^2 f)_{11}&= f_{xx},\\[4pt]
(\nabla^2 f)_{12}&= f_{xy},\\[4pt]
(\nabla^2 f)_{13}&= f_{xu}-\tfrac12\,a\,f_x,\\[4pt]
(\nabla^2 f)_{14}&= f_{xv}-\tfrac12\,a\,f_u\,\\[4pt]
(\nabla^2 f)_{22}&= f_{yy},\\[4pt]
(\nabla^2 f)_{23}&= f_{yu}-\tfrac12\,b\,f_x,\\[4pt]
(\nabla^2 f)_{24}&= f_{yv}-\tfrac12\,b\,f_y,\\[6pt]
(\nabla^2 f)_{33}&= f_{uu}
-\tfrac12\!\Big[(xa+yb+c)\,a+x a_u+y b_u+c_u\Big]\,f_x \\
&\hspace{4.1em}
-\tfrac12\!\Big[(xa+yb+c)\,b- x a_v- y b_v- c_v\Big]\,f_y
-\tfrac12\,(a\,f_u+b\,f_v),\\[6pt]
(\nabla^2 f)_{34}&= f_{uv}
-\tfrac12\,(x a_u+y b_u+c_u)\,(f_x+f_y),\\[6pt]
(\nabla^2 f)_{44}&= f_{vv}
-\tfrac12\!\Big[(xa+yb+c)\,a- (x a_u+y b_u+c_u)\Big]\,f_x\\
&\hspace{4.1em}
-\tfrac12\!\Big[(xa+yb+c)\,b+ (x a_v+y b_v+c_v)\Big]\,f_y
-\tfrac12\,(a\,f_u+b\,f_v).
\end{aligned}
$$
\end{remark}
\subsection{Ricci-Yamabe soliton}
In this Subsection, we recall the fundamental definitions and notions of Ricci-Yamabe solitons. 
 Let  $(M^n, g)$ be an  $n$-dimensional Riemannian manifold, then we defined on $M$ the \it{Ricci-Yamabe solitons} as a  self-similar solutions to \it{Ricci Yamabe flow} \cite{Gu} defined by:
\begin{equation} \label{eyq}
\left\{\begin{array}{ll}
\frac{\partial }{\partial t}g(t)&=-2\beta_1\Ric(t)+\beta_2 R(t) g(t),\\ g(0)&={g}_{0},
\end{array} \right.
\end{equation}
where $R$ is the scalar curvature of the Riemannian metric $g$, $\Ric$ is the Ricci curvature tensor of the metric,  and $\beta_2$ is a real constant.  When $\beta_2=0$ in ~\eqref{eyq}, then we get a Ricci flow. 

  A Riemannian manifold   $(N^m, g)$ is called   \it{Ricci Yamabe soliton} if the following equation is satisfied \cite{Gu}: 
  \begin{equation} \label{eqp2}
2\beta_1\Ric +\mathcal{L}_{X}g=(-2\lambda+\beta_2R)g,
\end{equation}
where $\Ric$ is the Ricci curvature tensor of $g$, $R$ is a scalar curvature, $\mathcal{L}_{X}$ denotes the Lie derivative operator along the vector field $X$, which is called soliton or potential, $\lambda$ and $\beta_i, i=1,2$  are real constants. 
 It will be denoted by  $(g, X, \lambda, \beta_1, \beta_2)$. 
If  $\lambda<0$, $\lambda=0$ and  $\lambda>0$,  then $(g, X, \lambda, \beta_1, \beta_2)$  is called  \it{expanding}, \it{steady} and  \it{shrinking}, respectively.   If $X$ denotes the gradient of a smooth function $f$ in a manifold $M$, then the above soliton is designated as \it{gradient Ricci-Yamabe soliton}. Consequently, equation \eqref{eqp2} reduces to
\begin{equation}\label{eqn}
2 \nabla^2 f+2\beta_1\Ric=\left(-2\lambda+\beta_2 R \right)g, 
\end{equation}
where $\nabla^2 f$ is the Hessian of $f$.
\section{Main Results}\label{sec:resultats}
\begin{theorem}
   Let $(M^4,g)$ be a Walker manifold of dimension $4$ and $X=\nabla f+Y$ with $\operatorname{div}(Y)=0$ and $f$ a smooth function in $M^4$.  Then we have $(M^4,g,X,\beta_1,\lambda, \beta_2)$ is a Ricci-Yamabe soliton if and only if  \[\Delta f=4\left(-\lambda+\left(\frac{\beta_2}{2}-\beta_1\right)\left(\partial xx f_1+\partial yy f_3 +2\partial x\partial y f_2\right)\right).\]
\end{theorem}
\begin{proof}
   Let $X = \nabla f + Y$ where $f$ is the De Rham potential and \(\operatorname{div}(Y)=0\). We have  \((M^4,g,X,\lambda,\beta_1,\beta_2)\) is Ricci-Yamabe soliton iff we have the following equation \[\beta_1\operatorname{Ric}+\frac{1}{2}\left(\mathcal{L}_{\nabla f+Y}\right)g=\left(-\lambda+\frac{\beta_2}{2}R \right)g\] so we have  \[\beta_1\operatorname{Ric}+\nabla^2 f +\frac{1}{2}\mathcal{L}_ Yg=\left(-\lambda+\frac{\beta_2}{2}R \right)g.\] Using the trace, we obtain \[\Delta f=4\left(-\lambda+\left(\frac{\beta_2}{2}-\beta_1\right)R \right)\] because \(\operatorname{tr}(\mathcal{L}_Yg)=0\) since \(\operatorname{div}(Y)=0.\)
\end{proof}
\begin{example}
    For the metric \(g\), we take \(f_1(x,y,u,v)=f_3(x,y,u,v)=K_1\), \(f_2(x,y,u,v)=K_2y^2\) and  \(\lambda=\left(\frac{\beta_2}{2}-\beta_1\right)K_2\) where \(K_1\) and \(K_2\)  are constants. Then for all smooth vector field \(X=\nabla f+Y\) in \(\mathcal{X}(M^4)\) such that \(\Delta f=0\) and \(\operatorname{div}(Y)=0\), we have \((M^4, g, X, \beta_1, \lambda,\beta_2)\) is Ricci-Yamabe soliton.
\end{example}
\begin{theorem}
   Let $(M^4,g)$ be a closed $4$-Walker manifold and $X=\nabla f+Y$ with $\operatorname{div}(Y)=0$ and $f$ a smooth function in $M^4$. If $(M^4, g, X, \beta_1, \lambda, \beta_2)$ is a Ricci-Yamabe soliton then it is steady or its scalar curvature is a constant.
\end{theorem}
\begin{proof}
    According to what has been said above, we have: \[\Delta f=4\left(-\lambda+\left(\frac{\beta_2}{2}-\beta_1\right)R \right)\] and Stokes' theorem implies that \(-\lambda+\left(\frac{\beta_2}{2}-\beta_1\right)R =0.\)
    \begin{itemize}
        \item If \(\beta_2=2\beta_1\) then we have \(\lambda=0\) and it is steady.
        \item If \(\beta_2\ne2\beta_1\) then we have \(R=\frac{2\lambda
        }{\beta_2-2\beta_1}\)
    \end{itemize}
    This complete the proof.
\end{proof}
\begin{proposition}
    Let \((M^4,\widetilde{g})\) be a Walker  $4$-manifold such that
\((M^4,\widetilde{g},X,\lambda,\beta_1,\beta_2)\) be a Ricci-Yamabe soliton. If \(\beta_1=0\) then \(X\) is a conform Killing vector field. Moreover if \(X\) is Killing vector field then \(\beta_1=0\) and \((M^4,\widetilde{g},X,\lambda,0,\beta_2)\) is a steady Ricci-Yamabe soliton.
\end{proposition}
\begin{proof}
Let  \(X\) be a smooth vector field and  \((M^4,\widetilde{g},X,\lambda,0,\beta_2)\) be a Ricci-Yamabe soliton. The equation becomes  \(\mathcal{L}_Xg=-2\lambda g\).\\
    Suppose now that \(\mathcal{L}_Xg=0\) and suppose also that \((M^4,\widetilde{g},X,\lambda,\beta_1,\beta_2)\) be a Ricci-Yamabe soliton. So the equation becomes 
    \(\beta_1\operatorname{Ric}=-2\lambda g\).\\
    Using the trace we get \(0=8\lambda\) which implies that \(\beta_1=0.\)
\end{proof}
\begin{theorem}\label{T1} Let \((M^4,\widetilde{g})\) be a Walker $4$-manifold.
Let's consider \(X=\nabla f+Y\) where \(f\) is the De Rham  potential function and \(\operatorname{div}(Y)=0\). Then \((M^4,\widetilde{g},X,\lambda,\beta_1,\beta_2)\) is a  Ricci-Yamabe soliton iff the Poisson equation \(\Delta f=-4\lambda \) admits at least one solution.
\end{theorem}

\begin{proof}
   Let \(X=\nabla f+Y\) where \(f\) is the De Rham  potential function and \(\operatorname{div}(Y)=0\) and suppose that  \((M^4,\widetilde{g},X,\lambda,\beta_1,\beta_2)\) is a Ricci-Yamabe soliton. We have the equation: \[\beta_1\operatorname{Ric}+\frac{1}{2}\left(\mathcal{L}_{\nabla f+Y}\right)g=-\lambda g\] so we get  \begin{eqnarray}\label{T3.5}
   \beta_1\operatorname{Ric}+\nabla^2 f +\frac{1}{2}\mathcal{L}_ Yg=-\lambda g.\end{eqnarray}
   Using the trace in \eqref{T3.5} and \(\operatorname{div}(Y)=0\),  we obtain \[\Delta f=-4\lambda\] since  \(\operatorname{tr}(\mathcal{L}_Yg)=0\).
\end{proof}
\begin{corollary}
If the components of the metric $g$ depend only $(u,v) $ i.e. $(M^4,g)$ is a strict Walker manifold then for any $Y$ vector field of zero divergence, $(M^4,g,Y+\nabla f,\lambda,  \beta_1, \beta_2)$ is a Ricci-Yamabe soliton with $$f(x,y,u,v)=-\lambda(xu + yv) + axv+byu+F(u, v),$$ $a, b$ of the real constants and $F(u,v)$ a smooth function dependent on $u$ and $v$.
\end{corollary}
\begin{proof}
    We have : \begin{align*}
\Delta f &= \left(-f_1 \frac{\partial^2 f}{\partial x^2} - 2f_2 \frac{\partial^2 f}{\partial x \partial y} - f_3 \frac{\partial^2 f}{\partial y^2} + 2 \frac{\partial^2 f}{\partial x \partial u} + 2 \frac{\partial^2 f}{\partial y \partial v}\right) \\
&- \left(\frac{\partial f_1}{\partial x} + \frac{\partial f_2}{\partial y}\right) \frac{\partial f}{\partial x} - \left(\frac{\partial f_2}{\partial x} + \frac{\partial f_3}{\partial y}\right) \frac{\partial f}{\partial y}.
\end{align*}
Since \((f_i)_{\{1;2;3\}}\) depend only of \(u\) and  \(v\) then we have  \begin{align*}
\Delta f &= -f_1 \frac{\partial^2 f}{\partial x^2} - 2f_2 \frac{\partial^2 f}{\partial x \partial y} - f_3 \frac{\partial^2 f}{\partial y^2} + 2 \frac{\partial^2 f}{\partial x \partial u} + 2 \frac{\partial^2 f}{\partial y \partial v}.
\end{align*}
It is enough to calculate the divergence of \(Y+\nabla f\).
\end{proof}
\begin{corollary}
    If the De Rham potential $f$ of the Theorem \ref{T1} depends on $u$ and $v$ then $(M^4,\widetilde{g},\lambda, \beta_1, \beta_2)$ is a steady Ricci Yamabe soliton.
\end{corollary}
For the proof, we can just notice that any function $f$ in $M^4$ such that $\frac{\partial f(x,y,u,v)}{\partial x}=\frac{\partial f(x,y,u,v)}{\partial y}=0$, is the solution of the Laplace-Beltrami equation.
\begin{example}
    For \(X=\nabla f+ Y\), we take the function given by \(f(x,y,u,v)=uv\) and \(Y=0\). We get \(\Delta f=0\).
\end{example}
\begin{theorem}
    Let \((M^4,\widetilde{g})\) be a Walker $4$-manifold and  \(X=\nabla f+Y\) with \(\operatorname{div}(Y)=0\), and let \(f\) be a smooth function in \(M^4\).   Then the manifold \(M^4\) is close and \(\Delta f=-4\lambda\) if and only if  \((M^4,\widetilde{g},\lambda,\beta_1,\beta_2)\) is a steady Ricci-Yamabe soliton. 
\end{theorem}
\begin{proof}
   Suppose that $M^4$ is close and $(M^4,\widetilde{g},\lambda, \beta_1, \beta_2)$ is a Ricci-Yamabe soliton, so according to Stokes' theorem we have \[0=\int_{M^4} \operatorname{div}(X)dM^4=\int_{M^4} \Delta f dM^4=-4\lambda\int_{M^4}dM^4=-4\lambda \int_{M^4}dxdy du dv=-4\lambda\operatorname{Vol}\left(M^4\right).\] This directly leads to \(\lambda=0.\)
\end{proof}
\begin{theorem}
    If the function $a(u,v)$ is non-zero then there is no function $f$ in $M^4$ such that $(M^4, \widetilde{g},  \nabla f, \beta_1, \lambda, \beta_2)$ is a gradien Ricci–Yamabe soliton.
\end{theorem}
\begin{proof} Suppuse that \(a(u,v)\ne 0\) and \((M^4,\widetilde{g}, \nabla f,\beta_1,\lambda, \beta_2)\) be a gradient Ricci-Yamabe soliton. By using the equations in \eqref{Hessien} we have
    
\begin{eqnarray*}
(\nabla^2 f)_{11}&=&(\nabla^2 f)_{12}=(\nabla^2 f)_{22}=(\nabla^2 f)_{14}=(\nabla^2 f)_{23}=0,\\
(\nabla^2 f)_{13}&=& f_{xu}-\tfrac12\,a\,f_x=-\lambda ,\\
(\nabla^2 f)_{24}&=& f_{yv}-\tfrac12\,b\,f_y=-\lambda,\\
(\nabla^2 f)_{33}&=&-\lambda \left(xa(u,v)+yb(u,v)+c(u,v)\right)-\beta_1\left(\frac{1}{2}b^2- b_v\right)
\end{eqnarray*}

\begin{eqnarray*}
(\nabla^2 f)_{34}&=& f_{uv}
-\tfrac12\,(x a_u+y b_u+c_u)\,(f_x+f_y)=-\beta_1\left(-\tfrac12\,a b+\tfrac12\,a_v+\tfrac12\,b_u\right),\\
(\nabla^2 f)_{44}&=& f_{vv}
-\tfrac12\!\Big[(xa+yb+c)\,a- (x a_u+y b_u+c_u)\Big]\,f_x\\
&&
-\tfrac12\!\Big[(xa+yb+c)\,b+ (x a_v+y b_v+c_v)\Big]\,f_y
-\tfrac12\,(a\,f_u+b\,f_v)\\
&&-\lambda \left(xa(u,v)+yb(u,v)+c(u,v)\right)-\beta_1\left(\tfrac12\,a^2-a_u\right).
\end{eqnarray*}

Since  \(\frac{\partial^2 f}{\partial ^2x}=0\) then we have \(f(x,y,u,v)=k_1(y,u,v) x+k_2(x,y,u,v)\). Moreover we have \(\frac{\partial^2 f}{\partial ^2y}=0\) so \[f(x,y,u,v)=\mathcal{A}_1(u,v)xy+\mathcal{A}_2(u,v)x+\mathcal{A}_3(u,v)y+\mathcal{A}_4(u,v).\]
Since  \(\frac{\partial ^2 f(x,y,u,v)}{\partial x\partial y}=0\), then \(f\) becomes  \[f(x,y,u,v)=\mathcal{A}_2(u,v)x+\mathcal{A}_3(u,v)y+\mathcal{A}_4(u,v).\] The equation $$f_{xu}-\tfrac12\,a\,f_x=-\lambda$$ becomes the equation \(\frac{\partial \mathcal{A}_2(u,v)}{\partial u}-\frac{a(u,v)}{2}\mathcal{A}_2(u,v)=-\lambda\)  and the homogeneous solution is \[\mathcal{A}^0_2(u,v)=\mathcal{A}(v)e^{\int_{\tau_0}^u\frac{a(s,v)}{2}ds}.\]
Suppuse now  \(\mathcal{A}^1_2(u,v)=\mathcal{K}(u,v)e^{\int_{\tau_0}^u\frac{a(s,v)}{2}ds}\) is a particular solution. So we have \[\frac{\partial \mathcal{K}(u,v)}{\partial u}=-\lambda e^{-\int_{\tau_0}^u\frac{a(s,v)}{2}ds}\Rightarrow \mathcal{K}(u,v)=-\lambda\int_{u_0}^u\left( e^{-\int_{\tau_0}^\tau\frac{a(s,v)}{2}ds}\right)d\tau.\] The general solution is : \[\mathcal{A}_2(u,v)=\left(\mathcal{A}(v)-\lambda\int_{u_0}^u\left( e^{-\int_{\tau_0}^\tau\frac{a(s,v)}{2}ds}\right)d\tau\right)e^{\int_{\tau_0}^u\frac{a(s,v)}{2}ds}.\] By analogy the equation \(f_{yv}-\tfrac12\,b\,f_y=-\lambda\), has a general solution as  \[\mathcal{A}_3(u,v)=\left(\mathcal{B}(u)-\lambda\int_{v_0}^v\left( e^{-\int_{\tau_1}^\tau\frac{b(u,t)}{2}dt}\right)d\tau\right)e^{\int_{\tau_1}^v\frac{b(u,t)}{2}dt}.\] We get   \( f_{xv}-\tfrac12\,a\,f_u=0\) and \(f_{yu}-\tfrac12\,b\,f_x=0\) so we have \[\begin{cases}
    \partial v \mathcal{A}_2(u,v)-\frac{a(u,v)}{2}\left(x\partial u\mathcal{A}_2(u,v)+y\partial u\mathcal{A}_3(u,v)+\partial u\mathcal{A}_4(u,v)\right)=0\\
    \partial u \mathcal{A}_3(u,v)-\frac{b(u,v)}{2}\mathcal{A}_2(u,v)=0
\end{cases}.\]
Since  \(a(u,v)\ne 0\) then we have \[\begin{cases}
    2\partial v \mathcal{A}_2(u,v)=a(u,v)\partial u\mathcal{A}_4(u,v)\\
    \partial u \mathcal{A}_2(u,v)=\partial u \mathcal{A}_3(u,v)=0\\
    \partial u \mathcal{A}_3(u,v)=\frac{b(u,v)}{2}\mathcal{A}_2(u,v)
\end{cases}.\]
\[0=\partial u\left(\left(\mathcal{A}(v)+\lambda\int_{u_0}^u\left( e^{-\int_{\tau_0}^\tau\frac{a(s,v)}{2}ds}\right)d\tau\right)e^{\int_{\tau_0}^u\frac{a(s,v)}{2}ds}\right).\]
We have \[F(v)=\left(\mathcal{A}(v)+\lambda\int_{u_0}^u\left( e^{-\int_{\tau_0}^\tau\frac{a(s,v)}{2}ds}\right)d\tau\right)e^{\int_{\tau_0}^u\frac{a(s,v)}{2}ds}.\] 
The functions \((u,v)\mapsto e^{\int\limits_{\tau_0}^u\frac{a(s,v)}{2}ds}\),  \((u,v)\mapsto  \int\limits_{u_0}^u\left( e^{-\int\limits_{\tau_0}^\tau\frac{a(s,v)}{2}ds}\right)d\tau\) always depend on the variable \(u\) because \((u,v)\mapsto a(u,v)\) is non zero. This means that the function \(v\mapsto F(v)\) depends of the variable \(u\), hence the absurdity.

\end{proof}
\begin{theorem}\label{t1}
Let \((M^4, \widetilde{g})\) be a Walker $4$-manifold where $\widetilde{g}_{ij}$ is given by \eqref{metric} with  $f_1(x, y, u, v)= yb(u,v)+c(u,v)$ and 
$f_3(x, y, u, v)=yb(u,v)+c(u,v)$,  \(b\) and \(c\) are smooth functions. Then 
     \((M^4,\widetilde{g}, \nabla f,\beta_1,\lambda, \beta_2)\) is a gradient Ricci-Yamabe soliton if and only if \[f(x,y,u,v)=\left(-\lambda u+ K\right)x+\left(\int_{u_1}^u\left(\frac{-\lambda s+K}{2}\right)b(s,v)ds+ h(v)\right)y+\mathcal{A}_4(u,v)\]   and \[\begin{cases}
         f_{uu}
-\tfrac12\!\big(y\,b_u+c_u\big)\,f_x
-\tfrac12\!\big((y b+c)\,b- y\,b_v- c_v\big)\,f_y
-\tfrac12\,b\,f_v
=-\lambda \big(yb(u,v)+c(u,v)\big)-\beta_1\!\left(\tfrac{1}{2}b^2-b_v\right),\\
f_{vv}
+\tfrac12\!\big(y\,b_u+c_u\big)\,f_x
-\tfrac12\!\big((y b+c)\,b+ y\,b_v+ c_v\big)\,f_y
-\tfrac12\,b\,f_v
=-\lambda \big(yb(u,v)+c(u,v)\big)\\
f_{uv}
-\tfrac12\!\big(y\,b_u+c_u\big)\,(f_x+f_y)
=-\tfrac{\beta_1}{2}\,b_u,\\
\left(\frac{-\lambda u+K}{2}\right)\partial vb(u,v)-\left(\frac{-\lambda u+K}{4}\right)b^2(u,v)-\partial u b(u,v)\left(\int_{u_1}^u\left(\frac{-\lambda s+K}{4}\right)b(s,v)ds+ h(v)\right)=0.
     \end{cases}\]
\end{theorem}
\begin{proof}
 \((M^4,\widetilde{g}, \nabla f,\beta_1,\lambda, \beta_2)\) is gradient Ricci-Yamabe soliton iff \[\beta_1 \operatorname{Ric}+\nabla^2 f=-\lambda \widetilde{g}.\]   This equation becomes :
$$
\begin{aligned}
(\nabla^2 f)_{11}&= f_{xx}=0,\\[4pt]
(\nabla^2 f)_{12}&= f_{xy}=0,\\[4pt]
(\nabla^2 f)_{13}&= f_{xu}=-\lambda,\\[4pt]
(\nabla^2 f)_{14}&= f_{xv}=0,\\[4pt]
(\nabla^2 f)_{22}&= f_{yy}=0,\\[4pt]
(\nabla^2 f)_{23}&= f_{yu}-\tfrac12\,b\,f_x=0,\\[4pt]
(\nabla^2 f)_{24}&= f_{yv}-\tfrac12\,b\,f_y=-\lambda,\\[6pt]
(\nabla^2 f)_{33}&= f_{uu}
-\tfrac12\!\big(y\,b_u+c_u\big)\,f_x
-\tfrac12\!\big((y b+c)\,b- y\,b_v- c_v\big)\,f_y
-\tfrac12\,b\,f_v\\
&=-\lambda \big(yb(u,v)+c(u,v)\big)-\beta_1\!\left(\tfrac{1}{2}b^2-b_v\right),\\[6pt]
(\nabla^2 f)_{34}&= f_{uv}
-\tfrac12\!\big(y\,b_u+c_u\big)\,(f_x+f_y)
=-\tfrac{\beta_1}{2}\,b_u,\\[6pt]
(\nabla^2 f)_{44}&= f_{vv}
+\tfrac12\!\big(y\,b_u+c_u\big)\,f_x
-\tfrac12\!\big((y b+c)\,b+ y\,b_v+ c_v\big)\,f_y
-\tfrac12\,b\,f_v
=-\lambda \big(yb(u,v)+c(u,v)\big).
\end{aligned}
$$
We have : \[\begin{cases}
    f(x,y,u,v)=\mathcal{A}_2(u,v) x+\mathcal{A}_3(u,v)y+\mathcal{A}_4(u,v)\\
    \mathcal{A}_2(u,v)=-\lambda u+ K \\
    \partial u\mathcal{A}_3(u,v)=\left(\frac{-\lambda u+K}{2}\right)b(u,v). 
\end{cases}\]  The equation \( f_{yv}-\tfrac12\,b\,f_y=-\lambda\) implies \( f_{yuv}-\tfrac12\,b\,f_{uy}-\tfrac12\,b_u\,f_{y}=0\) so we get  \[\left(\frac{-\lambda u+K}{2}\right)\partial vb(u,v)-\left(\frac{-\lambda u+K}{4}\right)b^2(u,v)-\partial u b(u,v)\left(\int_{u_1}^u\left(\frac{-\lambda s+K}{4}\right)b(s,v)ds+ h(v)\right)=0\] where \(K\) is a constant real.
\end{proof}
\begin{corollary}
For Theorem \ref{t1}, if the function $b$ is a null function, the function $c$ depends only on the variable $v$ and the function $h$ a smooth non-zero function then
    \[f(x,y,u,v)=(-\lambda u+K)x+h(v)y+\frac{a_1}{2}\left(u^2-v^2\right)+a_2 u+a_3v+a_4-2\lambda \int_\zeta ^v\left(\int_\xi^\tau c(t)dt\right)d\tau.\] Moreover \[c(v)=\left(m_1-\int_\varepsilon^v\left(\frac{2a_1}{h(\tau)}e^{2\lambda\int_{\tau_1}^\tau\frac{dt}{h(t)}}\right)d\tau\right) e^{-2\lambda\int_{\tau_1}^v\frac{dt}{h(t)}},\] where \(a_1\), \(a_2\), \(a_3\), \(a_4\), \(\varepsilon\), \(\zeta\), \(\xi\), \(\tau_1\) and \(m_1\) are constant reals.
\end{corollary}
\begin{proof}
 The constraints of Theorem \ref{t1} reduce to:
\[
\begin{cases}
\displaystyle
f_{uu}\;+\;\tfrac{1}{2}\,c_v\,f_y
\;=-\;\lambda\,c(v),\\[6pt]
\displaystyle
f_{vv}\;-\;\tfrac{1}{2}\,c_v\,f_y
\;=-\;\lambda\,c(v),\\[6pt]
\displaystyle
f_{uv}\;=\;0.
\end{cases}
\]
 But we have \(f(x,y,u,v)=(-\lambda u+K)x+h(v)y+\mathcal{A}_4(u,v)\) then the system becomes\[\begin{cases}
     \partial uu\mathcal{A}_4(u,v)+\frac{1}{2}c'(v) h(v)=-\lambda c(v)\\
     \partial vv\mathcal{A}_4(u,v)-\frac{1}{2}c'(v) h(v)=-\lambda c(v)\\
     \mathcal{A}_4(u,v)=\mathcal{U}(u)+\mathcal{V}(v)
 \end{cases}\]  so we obtain \(\mathcal{U}(u)=\frac{a_1}{2}u^2+a_2 u+a_3\) and \(c'(v)+\frac{2\lambda}{h(v)}c(v)=-\frac{2a_1}{h(v)}\).

 The differential equation \(c'(v)+\frac{2\lambda}{h(v)}c(v)=-\frac{2a_1}{h(v)}\) has a homogeneous solution \(c_0(v)=m_1e^{-2\lambda\int_{\tau_1}^v\frac{dt}{h(t)}}\). Suppose that  \(c_1(v)=m(v)e^{-2\lambda\int_{\tau_1}^v\frac{dt}{h(t)}}\)  as a particular solution of the differential equation. This implies \(m'(v)=-\frac{2a_1}{h(v)}e^{2\lambda\int_{\tau_1}^v\frac{dt}{h(t)}}\) then we have \[c(v)=\left(m_1-\int_\varepsilon^v\left(\frac{2a_1}{h(\tau)}e^{2\lambda\int_{\tau_1}^\tau\frac{dt}{h(t)}}\right)d\tau\right) e^{-2\lambda\int_{\tau_1}^v\frac{dt}{h(t)}}\]
 \[c'(v)=\frac{-2\lambda c(v)}{h(v)}-\frac{2a_1}{h(v)}\] \[\partial vv \mathcal{V}(v)=2\lambda c(v)-a_1\Leftrightarrow \mathcal{V}(v)=-\frac{a_1}{2}v^2+a_4v+a_5-2\lambda \int_\zeta ^v\left(\int_\xi^\tau c(t)dt\right)d\tau.\]
\end{proof}

\end{document}